\documentclass[a4paper,10pt]{amsart}

\title[Geometry on Kodaira surfaces]%
{Exotic geometric structures on Kodaira surfaces}
\date{\today}
 
\author{Benjamin McKay}

\address{School of Mathematical Sciences, 
University College Cork, Cork, Ireland}
\email{b.mckay@ucc.ie}

\thanks{It is a pleasure to thank
Sorin Dumitrescu for helpful conversations
on the problems solved in this paper,
and for his invitation to work with him at the
Laboratoire de Math{\'e}matiques J.A. Dieudonn{\'e} 
of the University of Nice Sophia--Antipolis
where this paper was written.}
\date{\today}

\usepackage[usenames,dvipsnames]{color}
\usepackage[latin1]{inputenc}
\usepackage{etex}
\usepackage{tikz}
\usetikzlibrary{matrix,arrows}
\usepackage{longtable}
\usepackage{ragged2e}
\usepackage{varioref}
\usepackage{amsfonts}
\usepackage{amsmath}
\usepackage{amssymb}
\usepackage{amsthm}
\usepackage{booktabs}
\usepackage[all]{xy}
\usepackage{fontenc}
\usepackage{ifthen}
\usepackage{braket}
\usepackage[pdftex]{hyperref}
\newtheorem{theorem}{Theorem}

\newtheorem{lemma}{Lemma}
\newtheorem{proposition}{Proposition}
\theoremstyle{remark}

\newtheorem{definition}{Definition}
\newtheorem{example}{Example}
\newtheorem{remark}{Remark}

\newcommand{\C}[1]{\ensuremath{\mathbb{C}^{#1}}}
\newcommand{\R}[1]{\ensuremath{\mathbb{R}^{#1}}}
\newcommand{\Z}[1]{\ensuremath{\mathbb{Z}^{#1}}}

\newcommand{\UpperHalfPlane}{\ensuremath{\mathbb{H}}}
\newcommand{\OO}[1]{%
  \ensuremath{%
    \mathcal{O}%
    \ifthenelse{\equal{#1}{0}}%
      {}%
      {\left({#1}\right)}%
  }%
}%
\newcommand{\OOp}[2]{
  \ensuremath{
    \mathcal{O}
    \ifthenelse{\equal{#1}{0}}
      {}
      {\left({#1}\right)}
    \ifthenelse{\equal{#2}{1}}
      {}
      {^{\oplus{#2}}}
  }
}
\newcommand{\Proj}[1]{\ensuremath{\mathbb{P}^{#1}}}

\newcommand{\LieSL}[1]{\ensuremath{\mathfrak{sl}\left(#1\right)}}
\newcommand{\PSL}[1]{\ensuremath{\mathbb{P}\operatorname{SL}\left(#1\right)}}

\newcommand{\Presentation}[2]{\Braket{#1|#2}}

\newcommand{\Cohom}[2]{\ensuremath{H^{#1}\left({#2}\right)}}
\DeclareMathOperator{\Ad}{Ad}

\newcommand{\pd}[2]{\frac{\partial #1}{\partial #2}}
\newcommand{\map}[3][:]%
{\ensuremath{\ifthenelse{\equal{#1}{:}}{}{{#1} \colon}{#2} \to {#3}}}
\newcommand{\mapto}[3][:]%
{\ensuremath{\ifthenelse{\equal{#1}{:}}{}{{#1} \colon}{#2} \mapsto {#3}}}

\newcommand{\homotopygp}[2]%
{\ensuremath{\pi_{#1}\left({#2}\right)}}
\newcommand{\fundgp}[1]%
{\ensuremath{\homotopygp{1}{#1}}}

\newcommand{\Lie}[1]{\ensuremath{\mathfrak{#1}}}
\newcommand{\LieG}{\Lie{g}}
\newcommand{\LieH}{\Lie{h}}

\newcommand{\disk}{\mathbb{D}}

\newcommand{\VCommutativeDiagram}[6]%
{%
\begin{tikzpicture}[description/.style={fill=white,inner sep=2pt},%
>=angle 90,
baseline=(current bounding box.center)]
\matrix (m) [matrix of math nodes, 
ampersand replacement=\&,
row sep=3em,
column sep=2.5em, text height=1.5ex, text depth=0.25ex]
{ #1 \& \& #2 \\
\& #3 \& \\ };
\path[->,font=\scriptsize]
(m-1-1) edge node[auto] {$ #4 $} (m-1-3)
edge node[auto] {$ #6 $} (m-2-2)
(m-1-3) edge node[auto] {$ #5 $} (m-2-2);
\end{tikzpicture}
}%
\newcommand{\BoxCommutativeDiagram}[8]%
{%
\begin{tikzpicture}[>=angle 90,baseline=(current bounding box.center)]
\matrix(m)[matrix of math nodes,
row sep=2.6em, column sep=2.8em,
ampersand replacement=\&,
text height=1.5ex, text depth=0.25ex]
{#1 \& #2 \\
#3 \& #4\\};
\path[->,font=\scriptsize,>=angle 90]
(m-1-1) edge node[auto] {$#5$} (m-1-2)
edge node[auto] {$#6$} (m-2-1)
(m-1-2) edge node[auto] {$#7$} (m-2-2)
(m-2-1) edge node[auto] {$#8$} (m-2-2);
\end{tikzpicture}%
}%
\newcommand{\BundleCommutativeDiagram}[5]%
{%
\begin{tikzpicture}[>=angle 90,baseline=(current bounding box.center)]
\matrix(m)[matrix of math nodes,
row sep=2.6em, column sep=2.8em,
ampersand replacement=\&,
text height=1.5ex, text depth=0.25ex]
{#1 \& #2 \\
 \& #3\\};
\path[->,font=\scriptsize,>=angle 90]
(m-1-1) edge node[auto] {$#4$} (m-1-2)
(m-1-2) edge node[auto] {$#5$} (m-2-2);
\end{tikzpicture}%
}%
\newcommand{\RightBundleCommutativeDiagram}[5]%
{%
\begin{tikzpicture}[>=angle 90,baseline=(current bounding box.center)]
\matrix(m)[matrix of math nodes,
row sep=2.6em, column sep=2.8em,
ampersand replacement=\&,
text height=1.5ex, text depth=0.25ex]
{#2 \& #1 \\
 #3 \& \\};
\path[->,font=\scriptsize,>=angle 90]
(m-1-2) edge node[auto] {$#4$} (m-1-1)
(m-1-1) edge node[auto] {$#5$} (m-2-1);
\end{tikzpicture}%
}%
\newcommand{\ShortExactSequence}[5]
{%
\begin{tikzpicture}[start chain,baseline=(current bounding box.center)] {
    \node[on chain] {$0$};
    \node[on chain] {$#1$} ;
    \node[on chain, join={node[above]
          {$\scriptstyle{#4}$}}] {$#2$};
    \node[on chain, join={node[above]
          {$\scriptstyle{#5}$}}] {$#3$};
    \node[on chain] {$0$}; }
\end{tikzpicture}
}%

\newcommand{\otoprule}%
{\midrule[\heavyrulewidth]\addlinespace[5pt]}

{%
\small
\begin{center}
\begin{longtable}{#1}
}%
{%
\end{longtable}
\end{center}
}%
\newcommand{\Aff}[1]%
{
\ensuremath{\operatorname{Aff}\left({#1}\right)}
}

\newcounter{remarkCounter}
\begin{document}

\begin{abstract}   
On all compact complex surfaces (modulo finite
unramified coverings), we classify all of the locally homogeneous 
geometric structures which are locally isomorphic to the
``exotic'' homogeneous surfaces of Lie.
\end{abstract}

\maketitle
\tableofcontents

\section{Introduction}

In Lie's classification of Lie group
actions on surfaces, there are two exotic
cases, in which the definition of the 
Lie group depends not
only on parameters, but on the 
set of solutions of a differential equation
\cite{Lie:GA:5} p .767--773, \cite{Mostow:1950}. We will classify,
on all compact complex surfaces, all of
the locally homogeneous structures
which are locally isomorphic to these
exotic surfaces of Lie.
This paper is part of a larger
programme to classify holomorphic 
locally homogeneous structures
on low dimensional compact 
complex manifolds,
in joint work with Sorin Dumitrescu and Alexey
Pokrovskiy \cite{McKay:2011,McKay/Pokrovskiy:2010}.

\begin{theorem}\label{1.5} 
Let $S$ be a compact complex surface. Suppose that
$S$ has a holomorphic 
locally homogeneous structure modelled on 
one of the exotic surfaces of Lie (defined precisely
below). Then,  up to replacing $S$ by a finite 
unramified covering
space, $S$ is a complex torus or primary Kodaira
surface. Every complex torus and every primary
Kodaira surface admits such structures. (We write
out these structures explicitly
in sections~\ref{section:GDtori}, \ref{section:GDprimeOntori},
\ref{section:GDonKodaira} and \ref{section:GDprimeOnKodaira}).
On any complex torus, all
holomorphic locally homogeneous 
structures modelled on Lie's exotic surfaces 
are induced by the translation structure.
(See section~\ref{subsection:Inducing} for this terminology.)

Let $G_0$ be the group of complex affine transformations
of $\C{2}$ of the form
\[
\begin{pmatrix}
z \\
w
\end{pmatrix}
\mapsto
\begin{pmatrix}
z + b \\
w + az + c
\end{pmatrix}.
\]
On any primary Kodaira surface, all holomorphic
locally homogeneous
structures modelled on Lie's exotic surfaces
are induced from a certain
holomorphic locally homogeneous structure
modelled on the $G_0$-action on $\C{2}$
which we will write out explicitly in 
section~\vref{section:GDonKodaira}.
\end{theorem}

\section{Definitions}

\subsection{\texorpdfstring{$\left(G,X\right)$-structures}%
{(G,X)-structures}}

Suppose that $G$ is a Lie group and that $H \subset G$
is a closed subgroup and let $X=G/H$.
\begin{definition}
An \emph{$X$-chart} on a manifold $M$
is a local diffeomorphism from an open
subset of $M$ to an open subset of $X$.
\end{definition}

\begin{definition}
Two $X$-charts $f_0$ and $f_1$ 
on a manifold are $G$-\emph{compatible}
if there is some element $g \in G$
so that $f_1 = g f_0$ wherever 
both $f_0$ and $f_1$ are
defined.
\end{definition}

\begin{definition}
An \emph{$X$-atlas} on a manifold $M$
is a collection of mutually $G$-compatible
$X$-charts whose domains cover $M$.
\end{definition}

\begin{definition}
A $\left(G,X\right)$-structure on a manifold $M$ 
is a maximal $(G,X)$-atlas. 
\end{definition}

\subsection{Pulling back}

\begin{definition}
If $\map[F]{M_0}{M_1}$ is a local diffeomorphism,
and $H \subset G$ a closed subgroup of a Lie group,
then every $(G,X)$-structure on $M_1$ has a \emph{pullback
structure} on $M_0$, whose charts are precisely
the compositions $f \circ F$, for $f$ a chart
of the $(G,X)$-structure. Conversely, if
$F$ is a normal covering map, and $M_0$
has a $(G,X)$-structure which is invariant
under the deck transformations, then
it induces a $(G,X)$-structure on $M_1$.
\end{definition}

\subsection{Developing maps and holonomy morphisms}

\begin{definition}
Suppose that $\left(M,m_0\right)$ 
is a pointed manifold,
with universal covering space 
$\left(\tilde{M},\tilde{m}_0\right)$.
Suppose that $H \subset G$ is a closed
subgroup of a Lie group. Let $X=G/H$
and $x_0 = 1 \cdot H \in X$.
A $(G,X)$-\emph{developing system}
is a pair $\left(h,\delta\right)$
of maps, where
\[
\map[\delta]%
{\left(\tilde{M},\tilde{m}_0\right)}%
{\left(X,x_0\right)}
\]
is a local diffeomorphism
and 
\[
\map[h]{\fundgp{M}}{G}
\]
is a group homomorphism
so that
\[
\delta\left(\gamma \tilde{m}\right)
=
h\left(\gamma\right)\delta\left(\tilde{m}\right),
\]
for every $\gamma \in \fundgp{M}$
and $\tilde{m} \in \tilde{M}$.
The map $\delta$ is called the
\emph{developing map}, and the
morphism $h$ is called the 
\emph{holonomy morphism}
of the developing system.
\end{definition}

\begin{definition}
Denote the universal covering map of a
pointed manifold $\left(M,m_0\right)$
as 
\[
\map[\pi_M]{\left(\tilde{M},\tilde{m}_0\right)}%
{\left(M,m_0\right)}.
\] 
Given a $(G,X)$-developing system $\left(h,\delta\right)$
on a manifold $M$, the \emph{induced $(G,X)$-atlas} 
on $M$ is the one whose charts consist of all maps $f$
so that $\delta = f \circ \pi_M$.
Every $(G,X)$-atlas lies in a unique $(G,X)$-structure,
so the induced $(G,X)$-structure is the $(G,X)$-structure
containing the induced $(G,X)$-atlas.
\end{definition}

\begin{remark}
Conversely, it is well known \cite{Goldman:2010} that 
if $G$ acts faithfully on $X=G/H$, then
every
$(G,X)$-structure is induced by a developing system
$\left(h,\delta\right)$, which is uniquely
determined up to (1) conjugacy:
\[
\mapto{\left(h,\delta\right)}%
{\left(\Ad(g)h,g\delta\right)}
\]
and (2) choice of a point $m_0 \in M$ to develop from.
\end{remark}

As a general principle, geometric objects 
(i.e. tensors, foliations, maps, etc.) on $(G,X)$ 
which are invariant under the image of the holonomy
morphism induce geometric objects of
the same type on $M$. 

\subsection{Inducing structures from other structures}%
\label{subsection:Inducing}
\begin{definition}
Suppose that 
$\map[h]{G_0}{G_1}$
is a morphism of Lie groups,
that $H_0 \subset G_0$
and $H_1 \subset G_1$
are closed subgroups.
Let $X_0=G_0/H_0$ and $X_1=G_1/H_1$.
Suppose that 
$h\left(H_0\right) \subset H_1$.
Define a smooth map
\[
\map[\delta]{X_0}{X_1}
\]
by
\[
\delta\left(g_0H_0\right)
=
h\left(g_0\right)H_1.
\]
We call $\left(h,\delta\right)$ a \emph{morphism} of 
homogeneous spaces. Suppose also that 
\[
\map[h'(1)]{\LieG_0/\LieH_0}{\LieG_1/\LieH_1}
\]
is a linear isomorphism.
Then clearly $\delta$ is a local diffeomorphism,
and we call $\left(h,\delta\right)$ an \emph{inducing} morphism
of homogeneous spaces.
If a manifold $M$ is equipped with an $X_0$-chart $f$,
then $\delta \circ f$ is clearly a $X_1$-chart.
A $\left(G_0,X_0\right)$-structure $\left\{f_{\alpha}\right\}$
has \emph{induced} $\left(G_1,X_1\right)$-structure
$\left\{\delta \circ f_{\alpha}\right\}$.
Every developing system $\left(h_0,\delta_0\right)$
on $M$ has induced developing system
$\left(h_1,\delta_1\right)=\left(h \circ h_0,\delta \circ \delta_0\right)$.
\end{definition}

\subsection{\texorpdfstring{The definition of $G_D$}%
{The definition of GD}}\label{subsec:GD}

Pick an effective divisor $D$ on $\C{}$ of positive degree. 
Let $p(z)$ be the monic
polynomial with zero locus $D$ (counting multiplicities).
Let $V_D$ be the set of all holomorphic functions 
\map[f]{\C{}}{\C{}} 
so that 
\[
p\left(\partial_z\right) f(z) = 0.
\]
If 
\[
D = n_1 \left[\lambda_1\right] + n_2 \left[\lambda_2\right]
+ \dots + n_{\ell} \left[\lambda_{\ell}\right],
\]
with distinct $\lambda_j$, then
the space $V_D$ has as basis the functions
\[
z^k e^{\lambda z}
\]
for $\lambda=\lambda_j$ and $0 \le k \le n_j-1$, $j=1,2,\dots,\ell$.
Let $G_D=\C{} \ltimes V_D$ with the group operation
\[
\left(t_0,f_0(z)\right)\left(t_1,f_1(z)\right)
=
\left(t_0+t_1,f_0(z)+f_1\left(z-t_0\right)\right),
\] 
and inverse operation
\[
\left(t,f(z)\right)^{-1}
=
\left(-t,-f\left(z+t\right)\right),
\]
and identity element $(0,0)$.

Let $G_D$ act on $\C{2}$ by the faithful group action
\[
\left(t,f\right)\left(z,w\right)
=
\left(z+t,w+f\left(z+t\right)\right).
\] 
The stabilizer of the origin of $\C{2}$ is
the group $H_D \subset G_D$ of pairs
$(0,f)$ so that $f(0)=0$. The surface $\C{2}$
with this action of $G_D$ is the \emph{exotic Lie surface
of the first kind}.

Writing elements of $\C{2}$ as $(z,w)$,
the action of $G_D$ preserves 
the holomorphic vector field $\partial_w$,
the holomorphic 1-form
$dz$, the foliation $z=\text{constant}$, 
the holomorphic volume form $dz \wedge dw$.
On each leaf of that foliation, the stabilizer
of the leaf preserves the holomorphic 1-form $dw$.

If we add more points to our divisor, i.e. 
add an effective divisor $D'$, clearly
\begin{align*}
V_D &\subset V_{D+D'}, \\
G_D &\subset G_{D+D'}, \\
H_D &\subset H_{D+D'}, \\
G_D/H_D &= G_{D+D'}/H_{D+D'} = \C{2}.
\end{align*}
So any $\left(G_D,\C{2}\right)$-structure is also a
$\left(G_{D+D'},\C{2}\right)$-structure.

Write the divisor $D$ as
\[
D = 
n_1 \left[\lambda_1\right] + 
n_2 \left[\lambda_2\right] + 
\dots
+ 
n_k \left[\lambda_k\right],
\]
where $n_j \in \Z{}_{> 0}$ is a multiplicity,
and $\lambda_j \in \C{}$ is a point.
Pick any nonzero complex number $\mu$.
Let
\[
D' = 
n_1 \left[\mu \lambda_1\right] + 
n_2 \left[\mu \lambda_2\right] + 
\dots
+ 
n_k \left[\mu \lambda_k\right].
\]
There is an obvious isomorphism
\[
\mapto{(t,f) \in G_D}%
{
\left(
  \frac{t}{\mu}, f\left(\mu z\right) 
\right) 
\in 
  G_{D'}
}
\]
and equivariant biholomorphism
\[
\mapto{\left(z,w\right) \in \C{2}}{\left(\frac{z}{\mu},w\right)}.
\]

The Lie algebra $\LieG_D$ of $G_D$ is
spanned by $\partial_z$ together with
all of the vector fields
\[
f(z) \, \partial_w 
\]
where $f \in V_D$. The adjoint action is
\begin{align*}
\left(t,f\right)_* \partial_z &= \partial_z + f'(z) \partial_w, \\
\left(t,f\right)_* g(z) \partial_w &= g(z-t) \partial_w. 
\end{align*}

\subsection{\texorpdfstring%
{The definition of $G'_D$}%
{The definition of GprimeD}}%
Let $G'_D = \C{} \times \C{\times} \times V_D$
with group operation
\[
\left(t_0,\mu_0,f_0(z)\right)
\left(t_1,\mu_1,f_1(z)\right)
=
\left(t_0+t_1,\mu_0 \mu_1, 
f_0(z) + \mu_0 \, f_1\left(z-t_0\right)\right),
\]
inverse operation
\[
\left(t,\mu,f(z)\right)^{-1}
=
\left(-t,\frac{1}{\mu},-\frac{f\left(z+t\right)}{\mu}\right),
\]
identity element
\[
(0,1,0),
\]
and action on $\C{2}$
\[
\left(t,\mu,f\right)(z,w)
=
\left(z + t, \mu w + f\left(z + t\right)\right).
\]
The stabilizer $H'_D$ of $(0,0) \in \C{2}$ is the
group of triples $\left(0,\mu,f(z)\right)$
for which $f(0)=0$.
The surface $\C{2}$
with this action of $G'_D$ is the \emph{exotic Lie surface
of the second kind}.

Writing elements of $\C{2}$ as $(z,w)$,
the action of $G'_D$ preserves 
the holomorphic 1-form
$dz$, the foliation $z=\text{constant}$.
On each leaf of that foliation, the stabilizer
of that leaf preserves the
holomorphic affine connection $\nabla=\partial_w$.
The action of $G_D'$ on $\C{2}$ preserves
the standard flat connection on the 
canonical bundle of $\C{2}$.

 The Lie algebra of $G'_D$ is generated by 
$\partial_z, w \partial_w$ 
and the vector fields $f(z) \partial_w$
for $f \in V_D$.
The adjoint action is
\begin{align*}
\left(t,\mu,f\right)_* \partial_z &= \partial_z + f'(z) \partial_w, \\
\left(t,\mu,f\right)_* w \partial_w &= \left(w-f(z)\right) \partial_w, \\
\left(t,\mu,f\right)_* g(z) \partial_w &= \mu g(z-t) \partial_w. 
\end{align*}

Clearly $G_D$ is a subgroup of $G'_D$.

\section{Statements of the theorems}

In the remainder of this paper we will prove the following theorems.

\begin{theorem}
For any effective divisor $D$ on $\C{}$,
either (1) $0$ does not lie in the support of $D$
and no complex torus admits any holomorphic
$\left(G_D,\C{2}\right)$-structure
or (2) 
$0$ lies in the support of $D$ and 
every complex torus admits a 1-parameter
family of pairwise nonisomorphic
holomorphic $\left(G_D,\C{2}\right)$-structures.
Any holomorphic $\left(G_D,\C{2}\right)$-structure 
on any complex torus is isomorphic to
precisely one of these.
\end{theorem}

\begin{theorem}
For any effective divisor $D$ on $\C{}$, 
and for each distinct point $\lambda$
in the support of $D$,
every complex torus admits a 1-parameter
family of pairwise nonisomorphic
holomorphic $\left(G'_D,\C{2}\right)$-structures
and an exceptional $\left(G'_D,\C{2}\right)$-structure.
Any holomorphic $\left(G'_D,\C{2}\right)$-structure 
on any complex torus is isomorphic to
precisely one of these.
\end{theorem}

\begin{theorem}
For any effective divisor $D$ on $\C{}$,
either (1) $0$ has multiplicity less than 2 in $D$
and no Kodaira surface admits any holomorphic
$\left(G_D,\C{2}\right)$-structure
or (2) 
$0$ has multiplicity at least $2$ in $D$ and
every primary Kodaira surface admits a 1-parameter
family of pairwise nonisomorphic
holomorphic $\left(G_D,\C{2}\right)$-structures.
Any holomorphic $\left(G_D,\C{2}\right)$-structure 
on any primary Kodaira surface is isomorphic to
precisely one of these.
\end{theorem}

\begin{theorem}
For any effective divisor $D$ on $\C{}$, 
and for each $\lambda$
with multiplicity 2 or more in $D$,
any primary Kodaira surface admits a 1-parameter
family of pairwise nonisomorphic
holomorphic $\left(G'_D,\C{2}\right)$-structures.
Any holomorphic $\left(G'_D,\C{2}\right)$-structure 
on any primary Kodaira surface is isomorphic to
precisely one of these.
\end{theorem}

\section{\texorpdfstring{$G_D$ and tori}{GD and tori}}%
\label{section:GDtori}

\begin{example}\label{example:TorusZeta}
Pick an effective divisor $D$ on $\C{}$.
Let $n_0$ be the order of $0$ in $D$,
and assume that $n_0 \ge 1$.
Let $G_0 = \left(\C{2},+\right)$, and
$H_0=\left\{0\right\} \subset G_0$.
Pick any $k \in \C{}$. Define a complex Lie group morphism
\[
\mapto[h]
{
\left(\lambda,\mu\right) \in \C{2}=G_0
}
{
\left(
\lambda,
\mu
+k
  \left(
    z^{n_0}-\left(z-\lambda\right)^{n_0}
  \right)
\right) 
\in G_D
}
.
\]
Check that this induces the biholomorphism
\[
\mapto[\delta]
{\left(s,t\right) \in \C{2} = G_0/H_0}
{\left(s,t+ks^{n_0} \right) \in \C{2}},
\]
so that $\left(h,\delta\right)$ is an inducing morphism of
homogeneous spaces.
Consequently every holomorphic $\left(G_0,\C{2}\right)$-structure, 
i.e. translation
structure on a complex surface, induces a holomorphic 
$\left(G_D,\C{2}\right)$-structure.
In particular, every complex torus bears a holomorphic
$\left(G_D,\C{2}\right)$-structure.
The inducing morphisms are conjugate
just when the induced $\left(G_D,\C{2}\right)$-structures are 
conjugate. We will prove 
below that all $\left(G_D,\C{2}\right)$-structures on
the torus are induced by an inducing morphism of 
this form, up to choice of affine coordinates
on the universal covering space of the torus. So the moduli 
space of $\left(G_D,\C{2}\right)$-structures modulo
conjugation is identified with $\C{}$, and does
not depend on the particular choice of complex torus.
\end{example}

\begin{proposition}\label{proposition:Torus}
Suppose that $S$ is a complex torus bearing a 
$\left(G_D,\C{2}\right)$-structure for some effective
divisor $D$ on $\C{}$. Then, for some
choice of affine coordinates on the universal
covering space of $S$, 
the $\left(G_D,\C{2}\right)$-structure is
one of those in example~\vref{example:TorusZeta}.
\end{proposition}
\begin{proof}
Pick any $\left(G_D,\C{2}\right)$-structure
on any complex torus $S=\C{2}/\Lambda$. Write the developing map
as
\[
\delta(s,t)=\left(z(s,t),w(s,t)\right)
\]
and the holonomy morphism as
\[
h\left(\lambda,\mu\right)
=
\left(t_{\lambda,\mu},f_{\lambda,\mu}\right),
\]
for each 
\[
\left(\lambda,\mu\right) \in \Lambda.
\]
So
\[
z\left(s+\lambda,t+\mu\right)
=
z(s,t)+t_{\lambda,\mu}.
\]
Consequently, $dz$ is a holomorphic 1-form on $S$, so
\[
dz = a \, ds + b \, dt
\]
for some complex numbers $a, b \in \C{2}$. We can
arrange, by linear change of variables on 
$\C{2}=\tilde{S}$ that $z=s$. Therefore
$t_{\lambda,\mu}=\lambda$. So $w(s,t)$ must satisfy
\[
w\left(s+\lambda,t+\mu\right)
=
w(s,t)+f_{\lambda,\mu}\left(s+\lambda\right),
\]
for every $\left(\lambda,\mu\right) \in \Lambda$
and $\left(s,t\right) \in \C{2}$.
So
\[
\pd{w}{t}\left(s+\lambda,t+\mu\right)
=
\pd{w}{t}(s,t),
\]
i.e. $\pd{w}{t}$ is a holomorphic function on
the compact complex surface $S$, so a constant,
which we rescale to be $1$, so
\[
w(s,t)=t+W(s),
\]
for some holomorphic function $W(s)$.
The holonomy equivariance of $\delta$ is then
precisely
\[
W\left(s+\lambda\right) - W(s) = f_{\lambda,\mu}(s+\lambda) - \mu.
\]
Let $P_D$ be the differential operator
\[
P_D = \prod_j \left(\partial_s - \lambda_j\right)^{n_j}.
\]
So a holomorphic function $F(s)$ lies in $V_D$ just 
when $P_D F(s)=0$. Applying $P_D$ to both sides,
\[
P_D W \left(s+\lambda\right) = P_D W \left(s\right) - P_D \mu.
\]
Therefore $P_D W'$ is a holomorphic function
on the surface $S$, and so is constant,
say
\[
P_D W'(s) = k_0,
\]
and
\[
P_D W(s) = k_0 s + k_1.
\]
Therefore
\[
k_0\left(s+\lambda\right) + k_1 = k_0 s + k_1 - P_D \mu.
\]
Clearly
\[
P_D \mu = \mu \prod_j \left(-\lambda_j\right)^{n_j}.
\]
So
\[
k_0 \lambda = - \mu \prod_j \left(-\lambda_j\right)^{n_j},
\]
for every $\left(\lambda,\mu\right) \in \Lambda$.
Since $\Lambda \subset \C{2}$ is a lattice, 
so spans $\C{2}$, we must have $k_0=0$ and 
\[
\prod_j \left(-\lambda_j\right)^{n_j} = 0,
\]
i.e. $\lambda_j = 0$ for some $j$, i.e.
$0$ lies in the support of $D$.
Now $P_D W(s) = k_1$ is a constant, so
\[
W(s) \in V_{D+[0]},
\]
i.e.
\[
W(s) = f(s) + ks^{n_0},
\]
for some $f \in V_D$. 
We can replace our developing system
$\left(h,\delta\right)$ by the equivalent
system $\left(ghg^{1},g\delta\right)$
for any $g \in G_D$, and thereby arrange
$f=0$, so
\[
\delta(s,t)
=
\left(s,t+ks^{n_0}\right).
\]
\end{proof}

\section{\texorpdfstring{$G'_D$ on tori}{GprimeD on tori}}%
\label{section:GDprimeOntori}

\begin{example}\label{example:TorusExp}
Suppose that $D$ is an effective divisor on $\C{}$.
Let $G_0 = \left(\C{2},+\right)$, and
$H_0=\left\{0\right\} \subset G_0$.
Define
\[
\delta \colon (s,t) \in \C{2} \to \left(s,e^t\right) \in \C{2}, 
\]
and
\[
h \colon \left(\lambda,\mu\right) \in G_0 \to 
\left(\lambda,e^{\mu},0\right) \in
G_D'.
\]
Then $\left(h,\delta\right)$ is an inducing morphism.
Consequently every holomorphic $\left(G_0,\C{2}\right)$-structure, i.e. translation structure on a complex surface, induces a holomorphic 
$\left(G'_D,\C{2}\right)$-structure.
In particular, every complex torus bears a holomorphic
$\left(G'_D,\C{2}\right)$-structure.
\end{example}

\begin{example}\label{example:TorusZetaPrime}
Suppose that $D$ is an effective divisor on $\C{}$.
Pick a complex number $a$ in the support of $D$.
Let $n_a$ be the order of $a$ in $D$.
Let $G_0 = \left(\C{2},+\right)$, and
$H_0=\left\{0\right\} \subset G_0$.
Pick any $k \in \C{}$.
For any $\left(\lambda,\mu\right) \in \C{2}$, let
\[
f_{\lambda,\mu}(z)
=
e^{az}
\left(
  \mu + k\left(z^{n_a}-\left(z-\lambda\right)^{n_a}\right)
\right).
\]
Define a morphism of complex Lie groups
\[
\mapto[h]%
{\left(\lambda,\mu\right) \in \C{2}=G_0}%
{
\left(
\lambda,e^{a \lambda},f_{\lambda,\mu}
\right) \in G'_D
}.
\]
This morphism induces the biholomorphism
\[
\mapto[\delta]
{
\left(s,t\right) \in \C{2}=G_0/H_0
}
{
\left(
	s,
	e^{as}
	\left(
		t + ks^{n_a}
	\right)
\right) 
\in 
\C{2}=G'_D/H'_D
}
\]
so that $\left(h,\delta\right)$ is an inducing morphism.
Consequently every holomorphic $\left(G_0,\C{2}\right)$-structure, i.e. translation structure on a complex surface, induces a holomorphic 
$\left(G'_D,\C{2}\right)$-structure.
In particular, every complex torus bears a 
$\left(G'_D,\C{2}\right)$-structure.
Two induced $\left(G'_D,\C{2}\right)$-structures
are conjugate if and only if they are induced
from conjugate inducing morphisms. We will prove 
below that all $\left(G'_D,\C{2}\right)$-structures on
the torus are induced by such an inducing morphism
(except those induced as in example~\vref{example:TorusExp}).
So the $\left(G'_D,\C{2}\right)$-structures
on any complex 2-torus $\C{2}/\Lambda$
are identified, modulo conjugation and
modulo the choice of affine coordinates on 
the universal covering space of the torus, with
\[
* \sqcup \bigsqcup_{a \in \operatorname{supp} D} V_D/\left(\frac{d}{dz}-a\right)V_D,
\]
a disjoint union of 1-dimensional complex vector spaces
and one point for example~\vref{example:TorusExp}.
(The same moduli space parameterizes the inducing morphisms
$\left(G_0,\C{2}\right) \to \left(G'_D,\C{2}\right)$.) Note that this moduli space is independent of the choice of 
complex 2-torus, but depends on the support
of the divisor $D$.
\end{example}

\begin{proposition}\label{proposition:GPrimeDtorus}
Suppose that $S$ is a complex torus bearing a 
$\left(G'_D,\C{2}\right)$-structure for some effective divisor
$D$ on $\C{}$. Then,
for some
choice of affine coordinates on the universal
covering space of $S$, 
 the $\left(G'_D,\C{2}\right)$-structure is
that of example~\vref{example:TorusExp} or
one of those in example~\vref{example:TorusZetaPrime}.
\end{proposition}
\begin{proof} 
Pick any $\left(G'_D,\C{2}\right)$-structure
on any complex torus $S=\C{2}/\Lambda$. Write the developing map
as
\[
\delta(s,t)=\left(z(s,t),w(s,t)\right)
\]
and the holonomy morphism as
\[
h\left(\lambda,\mu\right)
=
\left(t_{\lambda,\mu},g_{\mu},f_{\lambda,\mu}\right),
\]
for each 
\[
\left(\lambda,\mu\right) \in \Lambda.
\]
So
\[
z\left(s+\lambda,t+\mu\right)
=
z(s,t)+t_{\lambda,\mu}.
\]
Consequently, $dz$ is a holomorphic 1-form on $S$, so
\[
dz = a \, ds + b \, dt
\]
for some complex numbers $a, b \in \C{2}$. We can
arrange, by linear change of variables on 
$\C{2}=\tilde{S}$ that $z=s$. Therefore
$t_{\lambda,\mu}=\lambda$. So $w(s,t)$ must satisfy
\[
w\left(s+\lambda,t+\mu\right)
=
g_{\mu} w(s,t)+f_{\lambda,\mu}\left(s+\lambda\right),
\]
for every $\left(\lambda,\mu\right) \in \Lambda$
and $\left(s,t\right) \in \C{2}$.
So
\[
\pd{w}{t}\left(s+\lambda,t+\mu\right)
=
g_{\mu} \pd{w}{t}(s,t).
\]
If $h=\pd{w}{t}$, then $dh/h$ is a holomorphic
$1$-form on $S$, so 
\[
\frac{dh}{h} = a \, ds + b \, dt,
\]
say. Therefore
\[
h(s,t) = Ae^{as + bt},
\]
for some constant $A \in \C{\times}$. 
This forces
\[
g_{\lambda,\mu} = e^{a \lambda + b \mu}
\]
for every $\left(\lambda,\mu\right) \in \Lambda$.
Integrate:
\[
w(s,t)
=
e^{as}
\begin{cases}
W(s) + B e^{bt}, & \text{ if } b \ne 0,
\\
W(s) + B t, & \text{ if } b = 0,
\end{cases}
\]
for some constant $B \in \C{\times}$
and some holomorphic function $W(s)$.
We can rescale and translate $t$ to arrange
i.e.
\[
w(s,t)
=
e^{as}
\begin{cases}
W(s) + e^t, & \text{ or }
\\
W(s) + t. &
\end{cases}
\]

Suppose for the moment that $w(s,t)=e^{as}\left(W(s)+Be^t\right)$.
The holonomy equivariance of $\delta$ is then
precisely
\[
e^{a(s+\lambda)}
\left(
W\left(s+\lambda\right) - e^{\mu} W(s)
\right) = f_{\lambda,\mu}(s+\lambda).
\]
Let $P_D$ be the differential operator
\[
P_D = \prod_j \left(\partial_s - \lambda_j\right)^{n_j}.
\]
So a holomorphic function $F(s)$ lies in $V_D$ just 
when $P_D F(s)=0$. Applying $P_D$ to both sides,
\[
P_D e^{as} W \left(s+\lambda\right) = e^{\mu} P_D e^{as} W \left(s\right).
\]
Suppose that 
\[
D = \sum_j n_j \left[\lambda_j\right].
\]
Let
\[
D_a= \sum_j n_j \left[\lambda_j-a\right].
\]
Then for any holomorphic function $f(s)$,
\[
P_D \left(e^{as} f(s)\right)
=
e^{as} P_{D_a} f(s).
\] 
So
\[
P_{D_a} W\left(s+\lambda\right)
=
e^{\mu} P_{D_a}W(s).
\]
Let $F=P_{D_a} W$:
\[
F(s+\lambda)=e^{\mu} F(s),
\]
for every $\left(\lambda,\mu\right) \in \Lambda$.
So $F$ is a holomorphic section of a flat line
bundle on a complex torus.
Any holomorphic section of a degree $0$
line bundle must be everywhere $0$,
unless the bundle is holomorphically trivial,
in which case the holomorphic
sections are everywhere $0$ or 
everywhere nonzero. So 
$F$ is either
everywhere $0$ or everywhere
nonzero. 
If $F$ is everywhere nonzero,
then
\(
\frac{dF}{F}
\)
is a holomorphic 1-form,
so
\[
\frac{dF}{F} = p \, ds + q \, dt,
\]
for some constants $p, q \in \C{}$.
But $F=F(s)$ is independent of $t$, so
\[
\frac{dF}{F} = p \, ds.
\]
Therefore
\[
F(s) = C e^{ps},
\]
for some constant $C \ne 0$.
Therefore the holomorphic function
\[
e^{ps-t}
\]
is defined on $S$, and so must be constant,
a contradiction to our hypothesis that $F \ne 0$. 
Therefore $F=0$.

Continuing to suppose that $w(s,t)=e^{as}\left(W(s)+e^t\right)$,
we can say that $F=0$ i.e. $P_{D_a} W=0$.
By replacing the developing map and holonomy
by action of an element of $G_D$, we 
can arrange that $W(s)=0$, i.e. 
\[
\delta(s,t)
=
\left(s,e^t\right).
\]
The holonomy morphism must be
\[
h\left(\lambda,\mu\right)
=
\left(\lambda,e^{\mu},0\right).
\]

Next we can suppose that 
\[
\left(z(s,t),w(s,t)\right)
=
\left(s,e^{as}\left(W(s)+t\right)\right).
\]
The holonomy equivariance says that
\[
e^{a(s+\lambda)} \left(W(s+\lambda)+ \mu - W(s) \right)
=f_{\lambda,\mu}(s+\lambda).
\]
Taking $P_D$ of both sides,
\begin{equation}\label{equation:ForMu}
P_{D_a} W(s+\lambda)-P_{D_a} W(s) + P_{D_a} \mu = 0.
\end{equation}
As above, $P_{D_a} W'(s)$ is a holomorphic
function on $S$ so constant, say
\[
P_{D_a} W'(s) = k_0.
\]
So
\[
P_{D_a} W(s) = k_0 s + k_1
\]
for some constants $k_0, k_1 \in \C{}$.
Plugging this in to equation~\ref{equation:ForMu},
\[
0 = 
k_0 \lambda + (-1)^{\deg D}\prod_j \left(\lambda_j-a\right)^{n_j} \mu.
\]
This linear equation between $\lambda$ and $\mu$ holds
for every $\left(\lambda,\mu\right) \in \Lambda$.
Since $\Lambda \subset \C{2}$ is a lattice,
there is no linear relation satisfied by its elements,
so $k_0=0$ and $\lambda_j=a$ for some $j$.
Moreover $P_{D_a} W(s) = k_1$ for some constant $k_1$,
i.e. $P_{D_a+[0]}W(s)=0$.
Therefore
\[
W(s) = e^{-as} f(s) + ks^{n_a},
\]
for a unique $f \in V_D$, where $n_a$ is the order of $a$ in $D$.
We can conjugate the holonomy to arrange that $f=0$.
\end{proof}

\section{Primary Kodaira surfaces}%
\label{subsection:EllipticElliptic}

A \emph{primary Kodaira surface} is 
a compact complex surface $S$ of
odd first Betti number which occurs as
the total space of a principal bundle
\[
\BundleCommutativeDiagram{E_1}{S}{E_0}{}{}
\]
where $E_0$ and $E_1$ are elliptic curves.
The canonical bundle of $S$ is trivial
\cite{Barth/Hulek/Peters/VanDeVen:2004} p. 147. Up to finite covering, all
holomorphic elliptic curve fibrations over an elliptic
curve base are principal 
\cite{Barth/Hulek/Peters/VanDeVen:2004} p. 147.
Kodaira \cite{Kodaira:1964} p. 788 theorem 19 shows
that primary Kodaira surfaces 
have the form $S=\Gamma\backslash\C{2}$ where $\Gamma$ is a discrete group
of affine transformations acting properly discontinuously
without fixed points on $\C{2}$, preserving $dz_1 \wedge dz_2$.
The classification of Suwa \cite{Suwa:1975}
pp. 247--249 says that
a compact complex surface of the
form $\Gamma \backslash \C{2}$ is an elliptic fiber bundle
with nonzero first Chern class if and only if
it has first Betti number 3.  
The description $S=\Gamma \backslash \C{2}$ makes explicit
that $S$ has an affine structure. 
It is convenient to write each affine
transformation \mapto{z}{az+b} as
a matrix
\[
\begin{pmatrix}
a_{11} & a_{12} & b_1 \\
a_{21} & a_{22} & b_2 \\
0 & 0 & 1
\end{pmatrix}.
\]
In fact, Suwa
proves that we can arrange all of the elements of
$\Gamma$ to be affine transformations of the form
\[
g=
\begin{pmatrix}
1 & g_1 & g_2 \\
0 & 1 & g_3 \\
0 & 0 & 1
\end{pmatrix},
\]
after perhaps replacing $S$ by a finite covering space.
Every element of $\Gamma$ clearly preserves the
holomorphic foliation $dz_2=0$. 
The leaves of this foliation
quotient to become the fibers of the elliptic surface $S$.

Fillmore and Schueneman \cite{Fillmore/Scheuneman:1973} main theorem,
Vitter \cite{Vitter:1972} p. 238 and
D{\"u}rr \cite{Durr:2005} lemma 4.8
provide the following even more explicit description.
Suppose that $S$ is a primary Kodaira surface
\[
\BundleCommutativeDiagram{E_1}{S}{E_0}{}{}
\]
where $E_0$ and $E_1$ are elliptic curves.
The fundamental group of $S$ admits 
a presentation
\[
\fundgp{S} = \Gamma =
\Presentation{a, b, c, d}%
{c, d \text{ central}, \left[a,b\right]=c^r}
\]
where $r$ is a positive integer.
The kernel of the morphism $\fundgp{S} \to \fundgp{E_0}$
is precisely the center of $\fundgp{S}$.
The surface $S$ is explicitly the quotient
$\Gamma\backslash \C{2}$, where $a, b, c, d \in \Gamma$
act as the affine transformations
\[
a =
\begin{pmatrix}
1 & a_1 & 0 \\
0 & 1 & 1 \\
0 & 0 & 1
\end{pmatrix}, \
b =
\begin{pmatrix}
1 & b_1 & 0 \\
0 & 1 & b_3 \\
0 & 0 & 1
\end{pmatrix}, \
c =
\begin{pmatrix}
1 & 0 & c_2 \\
0 & 1 & 0 \\
0 & 0 & 1
\end{pmatrix}, \
d =
\begin{pmatrix}
1 & 0 & 1 \\
0 & 1 & 0 \\
0 & 0 & 1
\end{pmatrix},
\]
for some complex numbers $a_1, b_1, c_1, b_3$.
The equation $[a,b]=c^r$ forces precisely
$b_1 = a_1 b_3 - r c_2$.
Moreover we can arrange that $c_2$ and $b_3$
both lie in the upper half plane. (We can
even arrange that $c_2$ and $b_3$ both
lie in a given fundamental domain for the
modular group in the upper half plane.)

Conversely, for any choice of complex constants
$a_1, c_2, b_3$, if we set 
$b_1 = a_1 b_3 - r c_2$, and if
$c_2$ and $b_3$ both lie in the upper
half plane, then we can define $a,b,c,d$ and $\Gamma$
as above, and the surface $S=\Gamma\backslash\C{2}$
is a primary Kodaira surface.

\begin{example}\label{example:ExtraFreedom}
For our purposes, we will need to add in two
extra parameters to this family: every primary Kodaira surface
can be written in many ways as the
the quotient
$\Gamma\backslash \C{2}$, where $a, b, c, d \in \Gamma$
act as the affine transformations
\[
a =
\begin{pmatrix}
1 & a_1 & 0 \\
0 & 1 & a_3 \\
0 & 0 & 1
\end{pmatrix}, \
b =
\begin{pmatrix}
1 & b_1 & 0 \\
0 & 1 & b_3 \\
0 & 0 & 1
\end{pmatrix}, \
c =
\begin{pmatrix}
1 & 0 & c_2 \\
0 & 1 & 0 \\
0 & 0 & 1
\end{pmatrix}, \
d =
\begin{pmatrix}
1 & 0 & d_2 \\
0 & 1 & 0 \\
0 & 0 & 1
\end{pmatrix},
\]
for some complex numbers $a_1, a_3, b_1, b_3, c_2, d_2$,
with $a_1 b_3 - b_1 a_3 = r c_2$, where $a_3, b_3$ are
$\R{}$-linearly independent, and $c_2, d_2$ are
$\R{}$-linearly independent. This extra freedom will
allow us to normalize various constants appearing later on.
\end{example}

\section{%
\texorpdfstring%
{$G_D$ and primary Kodaira surfaces}%
{GD and primary Kodaira surfaces}%
}%
\label{section:GDonKodaira}

\begin{example}\label{example:Vitter}
Let $G_0$ be the group of all complex 
affine transformations of the form
\[
g=
\begin{pmatrix}
1 & g_1 & g_2 \\
0 & 1 & g_3 \\
0 & 0 & 1
\end{pmatrix}
\]
for any $g_1, g_2, g_3 \in \C{}$, and let $G_0$
act on $\C{2}$ as complex affine transformations
\[
\mapto{\left(z_1,z_2\right)}{\left(z_1+g_1 \, z_2 + g_2, z_2+g_3\right)}.
\]
Let $H_0 \subset G_0$ be the subgroup fixing the origin in $\C{2}$.
By construction, every primary Kodaira surface has a 
holomorphic $\left(G_0,\C{2}\right)$-structure with developing
map the identity $\C{2} \to \C{2}$ and 
and holonomy morphism the inclusion $\Gamma \to G_0$. 
We leave the reader to check
that the group $\Gamma=\left<a,b,c,d\right> \subset G_0$ is 
dense in the complex algebraic Zariski topology of $G_0$,
so that there is no complex analytic reduction
of structure group of this structure to any
complex algebraic proper subgroup of $G_0$.
\end{example}

\begin{remark}
The real Lie subgroup of $G_0$ consisting of the
matrices
\[
\begin{pmatrix}
1 & g_1 & g_2 \\
0 & 1 & -\sqrt{-1}\bar{g}_1 \\
0 & 0 & 1
\end{pmatrix}
\]
acts transitively on $\C{2}$, and contains a conjugate of
each subgroup $\Gamma$ defined
in example~\vref{example:ExtraFreedom}, 
so we can reduce the structure group $G_0$ this
real Lie group. 
\end{remark}

\begin{example}
Define a complex Lie group isomorphism
\[
\mapto[h]%
{%
g=
\begin{pmatrix}
1 & g_1 & g_2 \\
0 & 1 & g_3 \\
0 & 0 & 1
\end{pmatrix}
\in G_0
}%
{
\left(g_3,g_1\left(z-g_3\right)+g_2\right) \in G_{2[0]}
}
\]
and biholomorphism
\[
\mapto[\delta]%
{
\left(z_1,z_2\right) \in \C{2}=G_0/H_0
}
{
\left(z_1, z_2\right) \in \C{2}=G_{2[0]}/H_{2[0]}
}
\]
so that $\left(h,\delta\right)$ is an inducing morphism.
Any holomorphic $\left(G_0,\C{2}\right)$-structure on any complex surface,
induces a 
holomorphic $\left(G_{2[0]},\C{2}\right)$-structure.
\end{example}

\begin{example}
We generalize the last example. Pick any constant $k \in \C{}$.
Define a complex Lie group isomorphism
\[
\mapto[h]%
{%
g=
\begin{pmatrix}
1 & g_1 & g_2 \\
0 & 1 & g_3 \\
0 & 0 & 1
\end{pmatrix}
\in G_0
}%
{
\left(g_3,\left(g_1+kg_3\right)\left(z-g_3\right) + g_2 + \frac{k}{2}g_3^2\right).
\in G_{2[0]}
}.
\]
This morphism defines a biholomorphism
\[
\mapto[\delta]%
{
\left(z_1,z_2\right) \in \C{2}=G_0/H_0
}
{
\left(z_1+\frac{k}{2} z_2^2, z_2\right) \in \C{2}=G_{2[0]}/H_{2[0]}
}
\]
so that $\left(h,\delta\right)$ is an inducing morphism.
Any holomorphic $\left(G_0,\C{2}\right)$-structure on any complex surface 
has a 1-parameter family of deformations given
by varying $k$ above. Since $G_0$ lies in the complex affine
group, any holomorphic $\left(G_0,\C{2}\right)$-structure determines a
holomorphic affine structure. Vitter \cite{Vitter:1972}
proves that every holomorphic affine structure on a
primary Kodaira surface is among one
of these, depending only on the arbitrary
choice of constant $k \in \C{}$. 

So every primary Kodaira surface
admits a 1-parameter family of 
$\left(G_{2[0]},\C{2}\right)$-structures.
Vitter proves that the induced affine structures are distinct. 
Therefore they are distinct as $\left(G_{2[0]},\C{2}\right)$-structures,
and nonisomorphic on every primary Kodaira surface.
\end{example}

\begin{lemma}
Every holomorphic $\left(G_{2[0]},\C{2}\right)$-structure
on any primary Kodaira surface is isomorphic to one of those
constructed in example~\vref{example:Vitter},
for some values of the constants 
\[
a_1, a_3, b_1, b_3, d_2, r
\]
used to construct the primary Kodaira surface 
and for some constant $k$ used to construct the
structure.
\end{lemma}
\begin{proof}
Take the standard structure, say $\left\{f_{\alpha}\right\}$,
as defined in example~\ref{example:Vitter}, i.e. the 
one with $k=0$.
Then consider some other structure, say $\left\{g_{\beta}\right\}$.
By lemma~\vref{lemma:TransverseTranslation},
we will see that the induced foliations
of the two structures must be identical,
because the only holomorphic foliation of
any primary Kodaira surface is the fibration
by elliptic curves. Every holomorphic vector
field on any primary Kodaira surface is
a constant multiple of the vector field
generating the principal bundle action.
Therefore the holomorphic vector fields
corresponding to $\partial_w$ in $G_{2[0]}/H_{2[0]}$
must be constant multiples of one another.
Similarly the holomorphic 1-forms corresponding
to $dz$ in the model must agree up to constant
multiple, since they have to vanish on the leaves of 
the same fibration. We can also assume, following
Vitter's classification of affine structures
on primary Kodaira surfaces
(and perhaps after deforming
the other structure by precisely one of
the inducing morphisms in Vitter's deformation)
that the two structures induce
the same affine structure.

Pick $p, q \ne 0$ any complex constants.
We can change variables according to
\[
\left(Z,W\right)=\left(pz,qw\right)
\]
and transform $G_{2[0]}$ by the morphism
\[
\left(t',f'\right)=\left(pt,qf\left(z/p\right)\right).
\]
This will then alter the values of the constants
used to construct the primary Kodaira surface, by
\begin{align*}
a_1 &\mapsto \frac{q}{p} a_1, \\
b_1 &\mapsto \frac{q}{p} b_1, \\
c_2 & \mapsto q c_2, \\
d_2 & \mapsto q d_2, \\
a_3 & \mapsto p a_3, \\
b_3 & \mapsto p b_3.
\end{align*}
For suitable choice of $p$ and $q$
this rescaling will ensure that these holomorphic vector fields
and 1-forms match precisely, and require us only
to change the choices of the constants used
to define the primary Kodaira surface. So 
we can assume that these vector fields and
1-forms agree, and that the affine structure
agrees. Take $(z,w)$
a local chart for the standard structure,
and $(Z,W)$ a local chart for the other
structure. Since they have the same affine
structure, the matrix
\[
\begin{pmatrix}
\pd{Z}{z} & \pd{Z}{w} \\
\pd{W}{z} & \pd{W}{w}
\end{pmatrix}
\]
is a constant invertible matrix. Since
the 1-forms match, $dZ=dz$, so $Z=z+c$.
Since the vector fields match,
$\partial_W=\partial_w$, so 
$\partial_w W=1$. Therefore
\[
Z=z+c, W=w+az,
\]
which is a chart for the standard structure.
\end{proof}

\begin{example}\label{example:KodairaGD}
Pick any constant $k \in \C{}$,
and any integer $n \ge 2$.
Let
\[
(t,f) = \left(0,\frac{k}{n}z^n\right) \in G_{(n+1)[0]}.
\]
Define a complex Lie group homomorphism
\[
\map[h=\Ad(t,f)]%
{%
\left(t_1,f_1\right)
\in
G_{2[0]}
}%
{
\left(
	t_1,
	f_1(z)+\frac{k}{n}
		\left(
			z^n-
			\left(
				z-t_1
			\right)^n
		\right)
\right)
\in
G_{n[0]}
}.
\]
(N.B. the cancellation of the $z^n$ leading terms here
ensures that this morphism of complex Lie groups
is valued in $G_{n[0]}$, \emph{not} merely in $G_{(n+1)[0]}$.)
This morphism defines a biholomorphism
\[
\mapto[\delta=(t,f)]{\left(z,w\right) \in \C{2}=G_{2[0]}/H_{2[0]}}%
{\left(z,w+\frac{k}{n} z^n\right) \in \C{2}=G_{n[0]}/H_{n[0]}}
\]
so that $\left(h,\delta\right)$ is an inducing morphism
of homogeneous spaces.
Any holomorphic $\left(G_{2[0]},\C{2}\right)$-structure on any 
complex surface induces a 1-parameter family of holomorphic 
$\left(G_{n[0]},\C{2}\right)$-structures by varying the parameter $k$. 
Clearly this generalizes Vitter's construction
of affine structures from example~\vref{example:Vitter},
which is precisely the case $n=2$.
\end{example}

\begin{lemma}\label{lemma:PrimaryKodaira1.5}
Pick an effective divisor $D$ on $\C{}$.
Suppose that $S$ is a primary Kodaira
surface bearing a 
$\left(G_D,\C{2}\right)$-structure. 
Then the $\left(G_D,\C{2}\right)$-structure is
a unique one of those in example~\vref{example:KodairaGD},
up to holomorphic isomorphism.
\end{lemma}
\begin{proof}
Suppose that $S$ is a primary Kodaira
surface, with universal covering space
$\tilde{S}=\C{2}$, with coordinates
$(s,t)$, and each element
\[
\gamma
=
\begin{pmatrix}
1 & \gamma_1 & \gamma_2 \\
0 & 1 & \gamma_3 \\
0 & 0 & 1
\end{pmatrix} 
\in \fundgp{S}
\]
acts on $\C{2}$ as
\[
\gamma(s,t)=
\left(
s+\gamma_3,
t+\gamma_1 s+\gamma_2
\right).
\]
Suppose that $S$ has a $\left(G_D,\C{2}\right)$-structure
with developing map $\delta$ and holonomy morphism
$h$. Write $\delta=\left(z,w\right)$
and $h\left(\gamma\right)=\left(\tau_{\gamma},f_{\gamma}\right)$.
Then, for each $\gamma \in \fundgp{S}$, 
\[
z
\left(
s+\gamma_3,
t+\gamma_1 s+\gamma_3
\right)
=
z\left(s,t\right)+\tau_{\gamma}.
\]
Therefore
\(
\pd{z}{t}
\)
is a holomorphic function on $S$, so constant.
By rescaling the $t$ coordinate on $\tilde{S}$,
we can arrange that
\(
\pd{z}{t} = 1
\)
or 
\(
\pd{z}{t} = 0
\).

First, consider the case where 
\(
\pd{z}{t} = 1
\),
say
\[
z(s,t)=t+Z(s),
\]
for some holomorphic function $Z(s)$.
Equivariance of the developing map under the holonomy action 
implies
\[
Z\left(s+\gamma_3\right)
-Z(s)
=t_{\gamma} - \gamma_1 s - \gamma_2.
\]
Therefore
\(
Z''(s)
\)
is a holomorphic function on $S$, so constant,
so
\[
Z(s) = k_0 + k_1 s + k_2 s^2
\]
for some constants $k_0, k_1, k_2 \in \C{}$.
Equivariance now says
\[
2 k_2 \gamma_3 s + k_2 \gamma_3^2 + k_1 \gamma_3 
= 
t_{\gamma} 
- \gamma_1 s-\gamma_2.
\]
This holds for every $s \in \C{}$, so 
\[
2 k_2 \gamma_3 = \gamma_1 \text{ and } 
k_2\gamma_3^2 + k_1 \gamma_3 = t_{\gamma} - \gamma_2
\]
for every
\[
\gamma
=
\begin{pmatrix}
1 & \gamma_1 & \gamma_2 \\
0 & 1 & \gamma_3 \\
0 & 0 & 1
\end{pmatrix} 
\in \fundgp{S}.
\]
In particular, plugging in the generators
for $\fundgp{S}$ which we described in
example~\vref{example:ExtraFreedom},
we find
\[
2 k_2 a_3 = a_1
\]
and
\[
2 k_2 b_3 = b_1
\]
where $a_1, a_3, b_1, b_3$ are the
various constants appearing in our
construction of the primary Kodaira
surface in example~\ref{example:ExtraFreedom}.
This implies, again in terms of those constants,
\begin{align*}
rc_2 &=a_1 b_3 - b_1 a_3 \\
&= 0.
\end{align*}
But the constant $c_2$ has to be in the upper
half plane, a contradiction. Therefore
we can assume that $\pd{z}{t}=0$.

Write $z(s,t)=z(s)$. Equivariance
of the developing map under the 
action of the fundamental group
implies
\[
z
\left(s+\gamma_3\right)
=z(s)+t_{\gamma},
\]
so that $z'(s)$ is a holomorphic
function on the surface $S$, so
a constant, say $z(s)=k_0s+k_1$.
Since the developing map is a local
biholomorphism, $k_0 \ne 0$.
We can arrange by translating the
coordinates on $\tilde{S}$
that $z(s)=k_0s$. We can rescale
the affine coordinates $(s,t)$
as we like, perhaps changing the
expression of the fundamental
group as affine transformations, but
keeping the same form of a primary
Kodaira surface. Therefore we 
can arrange $k_0=1$, so $z(s)=s$.

Equivariance of the developing map
under the fundamental group action
is now expressed as the equation
\[
w
\left(
s+\gamma_3,
t+\gamma_1 s+\gamma_2
\right)
=
w(s,t)
+
f_{\gamma}\left(s+\gamma_3\right).
\]
Therefore
\(
\pd{w}{t}
\)
is a holomorphic function on $S$, so
a constant, and we can rescale the
$t$ variable to arrange that this
constant is $1$, say
\[
w(s,t)=t+W(s).
\]
Equivariance of the developing map
under the fundamental group action
is now expressed as the equation
\begin{equation}\label{equation:WhatsUp}
W\left(s+\gamma_3\right)-W(s) 
= 
f_{\gamma}\left(s+\gamma_3\right) 
- \gamma_1 s - \gamma_2.
\end{equation}

Let $n_0$ be the order of $D$ at $0$. 
Let $P_D$ be the constant coefficient
differential operator
\[
P_D = \prod_j \left(\partial_s - \lambda_j\right)^{n_j}
\]
with kernel $V_D$. Then clearly
\[
P_{D+2[0]} W(s+c) - P_{D+2[0]} W(s) = 0.
\]
Therefore $P_{D+2[0]}W$ is a holomorphic
function on $S$, so constant, i.e.
\[
P_{D+3[0]} W(s)=0.
\]
Therefore
\[
W(s) = f(s) + k_0 s^{n_0} + k_1 s^{n_0+1} + k_2 s^{n_0+2},
\]
for some constants $k_0, k_1, k_2 \in \C{}$
and some $f \in V_D$.
By conjugacy of the holonomy morphism
inside $G_D$, we can arrange that $f(s)=0$, so
\[
W(s) = k_0 s^{n_0} + k_1 s^{n_0+1} + k_2 s^{n_0+2}.
\]
We need
\[
W(s+c)-W(s)+\gamma_1 s+\gamma_2 \in V_D
\]
for every 
\[
\gamma
=
\begin{pmatrix}
1 & \gamma_1 & \gamma_2 \\
0 & 1 & \gamma_3 \\
0 & 0 & 1
\end{pmatrix} 
\in \fundgp{S},
\]
i.e.,
\[
\left(n_0+2\right)k_2 \gamma_3 s^{n_0+1}
+
\left(n_0+1\right)\gamma_3 
\left(k_1 + \frac{\left(n_0+2\right)}{2}k_2 \gamma_3\right) 
s^{n_0}
+\gamma_1 s+ \gamma_2 \in V_D.
\]

If $n_0 = 0$, then this expands to
\[
0 = 
\left(
2 k_2 \gamma_3 + \gamma_1
\right) s 
+ 
\left(
k_2 \gamma_3^2 + k_1 \gamma_3 + \gamma_2
\right)
\]
so that $2 k_2 \gamma_3 + \gamma_1=0$ for every 
\[
\gamma
=
\begin{pmatrix}
1 & \gamma_1 & \gamma_2 \\
0 & 1 & \gamma_3 \\
0 & 0 & 1
\end{pmatrix} 
\in \fundgp{S}.
\]
Examining generators, we find
\[
a_1 = -2 k_0 a_3 \text{ and } b_1 = -2 k_0 b_3
\]
which contradicts
\[
0 \ne rc_2 = a_1 b_3 - b_1 a_3.
\]
Therefore $n_0 \ge 1$.

If $n_0 = 1$, then the same equation expands to
require that
\[
3 k_2 \gamma_3 s^2 
+ 
\left(3 k_2 \gamma_3^2 + 2 k_1 \gamma_3 + \gamma_1\right) s
\]
must be constant.
Examining generators of $\fundgp{S}$, we find 
$k_2 = 0$ and
$a_1 = -2k_1 a_3$ and $b_1=-2k_1 b_3$ which
again contradicts the requirement
\[
0 \ne rc_2 = a_1 b_3 - b_1 a_3
\]
for the constants appearing in the generators
of the fundamental group of a primary Kodaira
surface. So $n_0 \ge 2$.

Plugging in that $n_0 \ge 2$, then equivariance
under the holonomy morphism expands to
\[
k_2 \left(n_0+2\right) \gamma_3 s^{n_0+1} 
+\left(n_0+1\right)\gamma_3
\left(k_1 + \frac{n_0+2}{2}k_2 \gamma_3\right) s^{n_0}
\in V_D
\]
modulo the terms in $V_D$, so we need $k_2=0$,
and then clearly $k_1=0$ as well. So finally
$W(s)=k_0s^{n_0}$.
\end{proof}

\section{%
\texorpdfstring%
{$G'_D$ and primary Kodaira surfaces}%
{G'D and primary Kodaira surfaces}%
}%
\label{section:GDprimeOnKodaira}

\begin{example}\label{example:GDPrimeOnKodaira}
As in example~\vref{example:Vitter}, write elements of $G_{2[0]}$ as 
\[
(t,f)=\left(g_3,g_1\left(z-g_3\right)+g_2\right).
\]
Pick an integer $n \ge 2$ and any two complex
numbers $k,\lambda \in \C{}$.
Consider the homomorphism of complex Lie groups
\[
\mapto[h]%
{\left(t,f\right) \in G_{2[0]}}%
{\left(t',\mu',f'\right) \in G'_{n[\lambda]}},
\]
where $t'=g_3, \mu'=e^{\lambda g_3}$ and
\[
f'\left(z\right)=
e^{\lambda z}
\left(g_1\left(z-g_3\right)+g_2
+
k\left(z^n-\left(z-g_3\right)^n\right)\right).
\]
This morphism takes $H_{2[0]}$ to $H'_{n[\lambda]}$ and thereby
induces the biholomorphism
\[
\mapto[\delta]%
{\left(z,w\right) \in \C{2}}%
{\left(
  z,
  e^{\lambda z}
  \left(
    w+k z^n
  \right)
\right) \in \C{2}},
\]
so that $\left(h,\delta\right)$ is an inducing morphism,
giving a 1-parameter family of $\left(G'_D,\C{2}\right)$-structures
(depending on $k$) for any effective divisor $D$ with multiplicity
at least $n$ at $\lambda$, and for any 
$\left(G_{2[0]},\C{2}\right)$-structure. In particular, all
at once we obtain families of these structures on all
primary Kodaira surfaces.
\end{example}

\begin{lemma}\label{lemma:GDPrimeOnKodaira}
Pick an effective divisor $D$ on $\C{}$.
Suppose that $S$ is a primary Kodaira
surface. Then every 
$\left(G'_D,\C{2}\right)$-structure on $S$ is
a unique one of those in example~\vref{example:GDPrimeOnKodaira},
up to holomorphic isomorphism.
\end{lemma}
\begin{proof}
Suppose that the developing map is $\delta=(z,w)$.
As in the proof of lemma~\vref{lemma:PrimaryKodaira1.5},
we can arrange $z(s,t)=s$. 
Write the holonomy morphism as
\[
h \colon \gamma \in \fundgp{S} \to 
\left(\tau_{\gamma},g_{\gamma},f_{\gamma}\right) 
\in G'_D.
\]
Then for each
\[
\gamma=
\begin{pmatrix}
1 & \gamma_1 & \gamma_2 \\
0 & 1 & \gamma_3 \\
0 & 0 & 1
\end{pmatrix}
\in \fundgp{S},
\]
we have
\[
w
\left(
s+\gamma_3,
t+\gamma_1 s+\gamma_2
\right)
=
g_{\gamma} w(s,t) 
+ 
f_{\gamma}\left(s+\gamma_3\right).
\]
Let $h=\pd{w}{t}$. Since $(z,w)=(s,w)$
is a local biholomorphism, $h\ne 0$.
Moreover,
\[
\frac{\pd{h}{t}}{h}
\]
is invariant under $\fundgp{S}$, so is 
a holomorphic function on $S$, so constant,
say
\[
\pd{h}{t} = k_0 h,
\]
and so
\[
h(s,t)=e^{k_0 t}H(s),
\]
for some nowhere vanishing holomorphic function $H$.
Invariance of $h$ implies
\[
H\left(s+\gamma_3\right) 
= 
g_{\gamma}
e^{-k_0\left(\gamma_1 s+\gamma_2 \right)}H(s).
\]
So 
\[
H'\left(s+\gamma_3\right) 
= 
g_{\gamma}
e^{-k_0\left(\gamma_1 s+\gamma_2\right)}
\left(-k_0 \gamma_1 H(s) + H'(s)\right).
\]
Dividing by $H\left(s+\gamma_3\right)$,
\[
\frac{H'\left(s+\gamma_3\right)}{H\left(s+\gamma_3\right)}
=
-k_0\gamma_1+
\frac{H'(s)}{H(s)}.
\]
It then follows that
\[
\left(\frac{H'}{H}\right)'
\]
is invariant under $\fundgp{S}$, so a constant,
say
\[
\left(\frac{H'}{H}\right)'=k_1,
\]
some $k_1 \in \C{}$. But then
\[
\frac{H'(s)}{H(s)}=k_1s+k_2,
\]
some $k_2 \in \C{}$. The equation
\[
\frac{
H'\left(s+\gamma_3\right)
}
{
H\left(s+\gamma_3\right)
}
=
-k_0\gamma_1+
\frac{H'(s)}{H(s)}
\]
tells us that
\[
k_1 \gamma_3 = -k_0 \gamma_1,
\]
for all 
\[
\gamma=
\begin{pmatrix}
1 & \gamma_1 & \gamma_2 \\
0 & 1 & \gamma_3 \\
0 & 0 & 1
\end{pmatrix}
\in \fundgp{S}.
\]
There is no linear relation between
the $\gamma_1$ and $\gamma_2$ components of $\gamma$
among the generators of $\fundgp{S}$,
as we see above, so $k_0 = k_1 = 0$.
Therefore
\[
\frac{H'(s)}{H(s)}=k_2,
\]
so that
\[
H(s) = e^{k_2s + k_3}
\]
for some constant $k_3 \in \C{}$.
Going backward,
\[
\pd{w}{t} = e^{k_2 s+ k_3},
\]
so 
\[
w(s,t) = t e^{k_2 s + k_3} + W(s),
\]
for some holomorphic function $W(s)$
and $g_{\gamma}=e^{k_2 \gamma_3}$.

The equation for $w$ forces
\[
e^{k_2\left(s+\gamma_3\right)+k_3}\left(\gamma_1 s+\gamma_2\right) 
+ 
W\left(s+\gamma_3\right)
=
e^{k_2 \gamma_3}W(s) 
+ 
f_{\gamma}\left(s+\gamma_3\right).
\]
Therefore if we let
\[
F = P_{D+2\left[k_2\right]}W,
\]
then
\[
F\left(s+\gamma_3\right)
= 
e^{k_2 \gamma_3}F(s).
\]
So $F$ is a section of a degree $0$ line bundle
on the elliptic curve $E_0$ (the base
of the elliptic fibration $E_1 \to S \to E_0$
of the primary Kodaira surface $S$), 
and therefore
either $F$ is everywhere $0$ or everywhere nonzero.

Suppose that $F$ is everywhere nonzero. 
Clearly $dF/F$ is a holomorphic 1-form
on the elliptic curve $E_0$, so 
$dF/F=k_4 ds$, for some constant $k_4 \in \C{}$,
i.e.
\[
F(s) = k_5 e^{k_4 s}.
\]
By the $\R{}$-linear independence
of the periods of the elliptic
curve, $k_4=k_2$, so 
\[
F(s) = k_5 e^{k_2 s}.
\]
Therefore
\[
P_{D+3\left[k_2\right]} W =0.
\]
Suppose that $n_0$ is the order of $k_2$ in $D$.
Then
\[
W(s) = f(s) 
+ 
e^{k_2 s}
\left( 
K_0 s^{n_0} + K_1 s^{n_0+1} + K_2 s^{n_0+2}
\right)
\]
for some $f \in V_D$ and constants $K_0, K_1, K_2 \in \C{}$. 
We can conjugate the holonomy
in $G_D$ to arrange that $f=0$, so
\[
W(s) = e^{k_2 s}\left( K_0 s^{n_0} + K_1 s^{n_0+1} + K_2 s^{n_0+2}
\right).
\]
Similarly if $F$ is everywhere $0$ then $W(s)$ has this form, but with
$K_2=0$.

Equivariance of the developing map under the
holonomy morphism then implies
\[
e^{k_3}\left(\gamma_1 s+\gamma_2\right)
+
\left(n_0+1\right)
\gamma_3
\left(
K_1
+
\frac{\left(n_0+2\right)K_2\gamma_3}{2}
\right)
s^{n_0}
+
K_2\left(n_0+2\right)\gamma_3
s^{n_0+1}
\in V_{D_{k_2}}
\]
for all 
\[
\gamma=
\begin{pmatrix}
1 & \gamma_1 & \gamma_2 \\
0 & 1 & \gamma_3 \\
0 & 0 & 1
\end{pmatrix}
\in \fundgp{S}.
\]

If $n_0=0$, then this says
\[
0=
\left(e^{k_3}\gamma_1 + 2 K_2 \gamma_3\right)s
+
e^{k_3b} + K_1 \gamma_3 + K_2 \gamma_3^2.
\]
So 
\[
e^{k_3} \gamma_1 + 2 K_2 \gamma_3 = 0
\]
for all 
\(
\gamma
\in \fundgp{S}.
\)
There can be no complex linear relation between
the $\gamma_1$ and $\gamma_3$ components of $\gamma \in \fundgp{S}$,
a contradiction, so $n_0 \ge 1$.

If $n_0=1$, then this says that
\[
3 K_2 \gamma_3 s^2 +
\left(e^{k_3} \gamma_1 + 2 K_1 \gamma_3 
+ 3 K_2 \gamma_3^2\right) s
\]
is constant,
for all 
\(
\gamma
\in \fundgp{S}.
\)
Therefore $K_2=0$ and again a linear relation
\[
e^{k_3} \gamma_1 + 2 K_1 \gamma_3 =0
\]
which is impossible. So $n_0 \ge 2$.
We can rescale the $t$ variable,
possibly changing the representation
of the fundamental group of our surface,
but still keeping it a primary Kodaira
surface, to arrange $k_3=0$.

Now assuming $n_0 \ge 2$, we easily see
that $K_1=K_2=0$.
\end{proof}

\section{Nonsingular holomorphic foliations on compact complex surfaces}

\begin{lemma}\label{lemma:EllipticFibration}
Suppose that a compact complex surface $S$ is
the total space of an elliptic fibration,
and that $S$ contains no rational curves.
Then $S$ has no singular fibers and $S$ is an 
elliptic fiber bundle. Up to replacing $S$ by
a finite covering space,
$S$ is a principal bundle 
\[
\BundleCommutativeDiagram{E}{S}{C}{}{}
\]
where $E$ is some elliptic curve
and $C$ is a compact complex curve.
Moreover the following are then equivalent:
\begin{enumerate}
\item
the first Chern class vanishes: $c_1(C,S)=0$,
\item
the first Betti number $b_1(S)$ is even,
\item
$S$ is K{\"a}hler,
\item
$S \to C$ admits a holomorphic flat connection,
\item
$S \to C$ is topologically trivial (though
perhaps not holomorphically trivial).
\end{enumerate}
\end{lemma}
\begin{proof}
The singular fibers must have lower
genus than the generic fiber, so must
be rational. 
The $j$-invariant of the elliptic
curve fibers is a holomorphic function
on the base, so must be constant.
The fibration must be a holomorphic
fiber bundle by the Grauert--Fischer 
local triviality theorem;
see Barth et. al. \cite{Barth/Hulek/Peters/VanDeVen:2004} p. 29.
The fiber bundle has transition
functions valued in the biholomorphism
group of the elliptic curve fibers.
This group is a finite extension
of the elliptic curve, so taking a finite
covering space we can reduce the structure
group to the elliptic curve. 
The equivalence of the various conditions
is well known; see \cite{Barth/Hulek/Peters/VanDeVen:2004} p. 145--149.
\end{proof}

\begin{lemma}[Klingler \cite{Klingler:1998}]%
\label{lemma:EllipticSurfacePresentation}
Suppose that
\[
\BundleCommutativeDiagram{E}{S}{C}{}{}
\]
is a principal bundle, $E$ is an elliptic curve
and $C$ is a compact complex curve of genus $g \ge 1$.
Suppose that $b_1(S)$ is odd. Then the universal
covering space $\tilde{S}$ of $S$
is biholomorphic to 
$\C{2}$ (if $g \ge 1$) or
$\UpperHalfPlane \times \C{}$ (if $g \ge 2$).
The fundamental group of $S$ admits a presentation
\[
\fundgp{S} = 
\Presentation{a_1, b_1, a_2, b_2, \dots, a_g, b_g, c, d}%
{c, d \text{ central}, \prod_{i=1}^g \left[a_i,b_i\right]=c^r}
\]
where $r$ is a positive integer, equal
to the Chern number of the bundle $S \to C$.
\end{lemma}

\begin{lemma}\label{lemma:ChernClassZero}
Suppose that $S$ is an elliptic fiber bundle,
with a nonsingular holomorphic foliation transverse
to its fibers. Then, after perhaps replacing
$S$ by a finite covering space, $S$ becomes
a principal elliptic fiber bundle and there is a unique holomorphic
flat connection for $S \to C$, say with connection form $\eta$,
so that the foliation is $\eta=0$.
Conversely, if $\eta$ is one such holomorphic
connection form, then every transverse
holomorphic foliation on $S$ has a unique 
expression as 
$\eta=p^* \xi$ where $\xi$ is an arbitrary
holomorphic $1$-form on $C$.
\end{lemma}
\begin{proof}
Replace $S$ with a finite covering space to
arrange that $S$ is principal, $E \to S \to C$.
We can average over the action of the
elliptic curve $E$, to find a transverse
foliation which is $E$-invariant, i.e.
a holomorphic connection.
Therefore the fibration has a holomorphic
connection as a principal bundle over the curve $C$.
So the first Chern class of the bundle must
vanish (see Atiyah \cite{Atiyah:1957}). 
Let $\eta$ be a holomorphic connection 1-form.
Define a holomorphic section $\xi \in \Cohom{0}{S,p^* \Omega^1_C}$
by $\eta(v)=\xi\left(p_* v\right)$ for any
tangent vector $v$ tangent to a leaf of the foliation.
Since $\eta$ is a connection, $E$-invariance
of $\eta$ forces $\xi$ to be the pullback
of a 1-form on $C$, which we also call $\xi$.
Replace $\eta$ by $\eta-\xi$, which is also
a holomorphic connection.
\end{proof}

\begin{definition}
A nonsingular holomorphic foliation $F$ of a complex surface $S$
is \emph{turbulent} if there is an elliptic fibration
$S \to B$ so that finitely many fibers are tangent to $F$,
while all others fibers are everywhere transverse to $F$.
\end{definition}

\begin{lemma}\label{lemma:Brunella}[Brunella \cite{Brunella:1997}]
Every holomorphic nowhere singular foliation on any 
compact complex surface is one of the following:
\begin{enumerate}
\item a holomorphic fibration by compact complex curves,
\item everywhere transverse to an elliptic or rational fibration,
\item a turbulent foliation on an elliptically fibered surface,
\item a linear foliation on a complex torus,
\item a holomorphic foliation of a Hopf surface, linear or nonlinear (see Mall \cite{Mall:1998}),
\item a holomorphic foliation of an Inoue surface (see 
Inoue \cite{Inoue:1974},
Kohler \cite{Kohler:1995,Kohler:1996}),
\item a transversely hyperbolic foliation with
dense leaves on a surface whose universal covering
space is a holomorphic disk fibration over a disk. 
\end{enumerate}
\end{lemma}
These categories are not quite exclusive. For example,
on a product of compact Riemann surfaces,
the ``horizontal'' fibration 
is transverse to the obvious ``vertical'' fibration.

\begin{proposition}%
\label{proposition:EllipticFiberBundleFoliations}
Suppose that $S$ is
the total space of an elliptic fiber bundle
over a curve of genus $g \ge 2$.
Suppose that $b_1(S)$ is odd. Any nonsingular holomorphic foliation 
on $S$ either is turbulent or coincides
with the fibration.
\end{proposition}
\begin{proof}
Brunella's classification 
shows us that every nonsingular holomorphic foliation on $S$
either (1) is turbulent or (2) coincides with the fibration
or (3) is everywhere transverse to the fibration.
Suppose that the foliation is everywhere transverse.
By lemma~\vref{lemma:ChernClassZero},
we can suppose, after replacing $S$ with a finite covering
space of $S$, that $S$ has a holomorphic flat connection.
But the first Chern class does not
vanish by lemma~\vref{lemma:EllipticSurfacePresentation}.
\end{proof}

\begin{definition}
A \emph{holomorphic foliation with transverse translation structure}
on a complex surface
is a closed nowhere zero holomorphic 1-form.
The associated foliation is the one on whose
tangent lines the 1-form vanishes.
\end{definition}

\begin{lemma}\label{lemma:TransverseTranslation}
Suppose that $S$ is a compact complex
surface, containing no rational curves.
Suppose that $S$ admits a holomorphic nonsingular
foliation with a transverse translation
structure. Then up to replacing $S$ with
a finite covering space, $S$ is
\begin{enumerate}
\item
a complex torus
with a linear foliation, or 
\item
a principal elliptic bundle with a holomorphic
flat connection $\eta$ with foliation $\eta=0$, or
\item
a primary Kodaira surface, i.e. a holomorphic
principal bundle
\[
\BundleCommutativeDiagram{E_1}{S}{E_0}{}{}
\]
with elliptic curve structure group over an 
elliptic curve base, with foliation equal
to the fibration.
\end{enumerate}
\end{lemma}
\begin{proof}
Suppose that $S$ is a compact complex
surface with nowhere vanishing closed
holomorphic 1-form $\eta$, and consider
the foliation $\eta=0$.
Our foliation must be on Brunella's list
in lemma~\vref{lemma:Brunella}.

Every holomorphic 1-form on a complex
torus is translation invariant, and therefore
any foliation of the torus with transverse
translation structure must be a linear foliation.
So we can assume that $S$ is not a torus.

If $S$ is a fibration, and the foliation 
coincides with the fibration, then the 1-form is 
semibasic and so basic, and therefore arises
from a nowhere vanishing 1-form on the base,
so the base is an elliptic curve. 
The fibration must be a smooth fiber bundle,
by lemma~\vref{lemma:EllipticFibration}.
If the fibers are rational, then
$S$ contains a rational curve, contradicting
our hypotheses.
If the fibers are elliptic curves, then
the classification of elliptic
curve bundles over an elliptic curve base 
(see \cite{Barth/Hulek/Peters/VanDeVen:2004} p. 146)
tells us that (after perhaps replacing $S$ by
a finite covering space)
$S$ is a torus (and therefore
the foliation is linear) or a primary Kodaira
surface. If the fibers are curves of genus $g \ge 2$,
then replacing $S$ by a finite covering space,
we can arrange that $S=E \times C_{g \ge 2}$,
and the fibration is the trivial product fibration
(see \cite{Barth/Hulek/Peters/VanDeVen:2004} p. 149). 

So from now on we can suppose that the foliation does not
coincide with a fibration.
Suppose that $S$ is an elliptic fiber bundle,
$E \to S \to C$, 
and the holomorphic nowhere vanishing
1-form is $\eta$. On each fiber $S_c \cong E$,
for any $c \in C$, we have a distinguished Maurer--Cartan
1-form, say $dw$, for $w$ any affine coordinate
on the universal covering space of $E$.
On $S_c$, 
\[
\left.\eta\right|_{S_c} = f \, dw,
\]
for some holomorphic function \map[f]{S}{\C{}}.
But then $f$ must be constant. If $f=0$,
then the foliation
coincides with the fiber bundle, a
case already covered. So 
without loss of generality, 
\[
\left.\eta\right|_{S_c} = dw,
\]
over every point $c \in C$ and the
foliation is transverse. Apply 
lemma~\vref{lemma:ChernClassZero}
to find that $S$ is, after replacement by a finite
covering space, a holomorphic
elliptic fiber bundle with 
holomorphic flat connection $\eta$
and, without loss of generality,
the foliation is $\eta=0$.

No Hopf surface carries any nonzero holomorphic
1-form, nor does any Inoue surface
(see \cite{Barth/Hulek/Peters/VanDeVen:2004} p. 172 
for the proof for linear Hopf surfaces, and p. 176
for the proof for Inoue surfaces. The argument
on p. 172, without alteration, proves the 
result also for nonlinear Hopf surfaces.)

Suppose that $S$ has a transversely hyperbolic foliation with
dense leaves, and $\tilde{S} \to S$, the universal covering
space, is a holomorphic disk fibration over a disk:
\[
\BundleCommutativeDiagram{\disk}{\tilde{S}}{\disk}{}{}.
\]
The holomorphic 1-form on $S$, say $\eta$, must lift to
$\tilde{S}$ to be closed and 
vanish precisely on the leaves of the foliation.
The fibers are simply connected, so there is no monodromy
and $\eta$ must be pulled back from a 1-form
on the base disk, which we also call $\eta$. The
group $\fundgp{S}$ must preserve the form $\eta$
and the fibration, and therefore must act on the
base disk.
Let $\pi=\fundgp{S}$.
The function $\left|\eta\right|^2$ (using the hyperbolic
metric on the base disk) is a $\pi$-invariant 
smooth function on the base disk. It lifts to a $\pi$-invariant 
smooth function on $\tilde{S}$, constant on each leaf.
Therefore it descends to a smooth function on $S$, constant 
on each leaf. But this function must therefore be constant.
So $\eta$ has constant norm in the hyperbolic metric
on the base disk.
The hyperbolic metric is
\[
ds^2 = \frac{\left|dz\right|^2}{\left(1-|z|^2\right)^2},
\]
so $\eta$ must have the form
\[
\eta = \frac{e^{i \alpha}}{1-|z|^2} dz,
\]
(after perhaps replacing $\eta$ by a real constant
rescaling) 
for some real valued smooth function $\alpha$.
Since $\eta$ is closed, we can take exterior derivative to find
\begin{align*}
0 &= \pd{}{\bar{z}} \frac{e^{i\alpha}}{1-|z|^2}, \\
&= \frac{e^{i\alpha}}{\left(1-|z|^2\right)^2}
\left(
    i\left(1-|z|^2\right) \pd{\alpha}{\bar{z}} + z 
\right).
\end{align*}
Therefore
\[
\pd{\alpha}{\bar{z}} = \frac{iz}{1-|z|^2}.
\]
Since $\alpha$ is real, 
\[
\pd{\alpha}{z} = -\frac{i\bar{z}}{1-|z|^2}.
\]
But then
\[
d \alpha = \frac{i}{1-|z|^2} \left( \bar{z} \, dz -\bar{z} \, d\bar{z}\right).
\]
Taking exterior derivative, 
\[
d^2 \alpha = 
\frac{2i}{\left(1-|z|^2\right)^2} dz \wedge d\bar{z}.
\]
But $d^2=0$, so a contradiction.
Therefore there is no holomorphic 1-form
of constant nonzero norm
on the disk. Therefore there is no foliation on $S$ 
with holomorphic transverse translation structure.
\end{proof}

\section{Nothing else}

\begin{lemma}\label{lemma:noRationalCurves}
Let $S$ be a compact complex surface bearing a 
$\left(G_D,\C{2}\right)$-structure or a 
$\left(G'_D,\C{2}\right)$-structure. 
Then $S$ contains no rational curves.
\end{lemma}
\begin{proof}
Suppose that $S$ contains a rational curve.
The complete classification of all
holomorphic Cartan geometries on any
compact complex surface containing a
rational curve appears
in \cite{McKay:2011}. 
The connected complex homogeneous
spaces $\left(G,X\right)$ 
which appear as model spaces for those geometries
(with $G$ a connected complex 
Lie group acting holomorphically and 
effectively) are of three distinct types:
\begin{enumerate}
\item
$\left(\PSL{3,\C{}},\Proj{2}\right)$,
\item
$\left(\left(\Z{}/2\Z{}\right) \ltimes
\left(
\PSL{2,\C{}} \times \PSL{2,\C{}}
\right), \Proj{1} \times \Proj{1}\right)$,
\item
products $\left(G_0 \times \PSL{2,\C{}}, X_0 \times \Proj{1}\right)$.
\end{enumerate}
We want to claim that
$\left(G_D,\C{2}\right)$ and 
$\left(G'_D,\C{2}\right)$ do
not appear in that classification,
for any divisor $D$. Clearly
$G_D$ and $G'_D$ are connected,
and have normal subgroup $V_D$,
so not simple. Therefore we only
have to prove that $G_D$ and
$G'_D$ are not of the form
$G_0 \times \PSL{2,\C{}}$.
The Lie algebras have
\[
\left[\LieG_D, \LieG_D\right]=V_D, \ \left[V_D,V_D\right]=0,
\]
and similarly
\[
\left[\LieG'_D, \LieG'_D\right]=V_D, \ \left[V_D,V_D\right]=0,
\]
so that neither $\LieG_D$ nor $\LieG'_D$ contain
any copy of $\LieSL{2,\C{}}$.
\end{proof}

We prove theorem~\vref{1.5}.
\begin{proof} 
By lemma~\vref{lemma:noRationalCurves}, 
we can assume that $S$ contains no rational curves.
The surface $S$ inherits a nowhere vanishing
holomorphic closed 1-form $\omega$ (coming from $dz$),
and this 1-form is a transverse translation structure. 
The classification given in 
lemma~\vref{lemma:TransverseTranslation}
of those foliations shows that, up to replacement
by a finite covering space, $S$ is either 
\begin{enumerate}
\item 
a complex torus with linear foliation, or
\item 
a holomorphic prinicipal bundle of 
elliptic curves $E \to S \to C$
with a holomorphic flat connection, or
\item
$S$ is a primary Kodaira surface, i.e.
the total space of an elliptic fiber bundle
over an elliptic curve base, and the bundle 
has nonzero first Chern class, and $S$ has
odd first Betti number; the foliation is the
fibration.
\end{enumerate}

Suppose that $S \to C$ is a principal 
elliptic curve bundle with holomorphic
flat connection $\eta$.
We only need to prove that $C$ is an elliptic curve,
so that $S$ is a complex torus or primary Kodaira
surface.
Suppose that $S$ has a holomorphic 
$\left(G_D,\C{2}\right)$-structure.
But $G_D$ preserves a holomorphic
volume form on $\C{2}$, so $S$ is equipped
with a holomorphic volume form, which easily
implies that $C$ is an elliptic curve and $S$ is a complex torus.

Therefore we can assume that $S$ has a 
$\left(G'_D,\C{2}\right)$-structure.
If $C$ has genus $0$, then $S$ contains a rational
curve, contradicting our assumptions.
If $C$ has genus $1$, then $S$ is a
complex torus and
the relevant structures are
classified in sections~\ref{section:GDtori} and~\vref{section:GDprimeOntori}.
Therefore we can assume that $C$ has genus at least $2$.

The group $G'_D$ acts on $\C{2}$ preserving 
the holomorphic 2-form $dz \wedge dw$ up
to constant rescaling, so preserving a 
holomorphic connection on the canonical
bundle of $S$. Therefore the canonical
bundle has vanishing Atiyah class \cite{Atiyah:1957}, and
if $S$ is K\"ahler then $c_1(S)=0$.
By lemma~\vref{lemma:EllipticFibration} 
the bundle $S \to C$ is topologically
trivial: $c_1(S)=c_1(E) + c_1(C)=c_1(C) \ne 0$.
Therefore $S$ has no $\left(G'_D,\C{2}\right)$-structure.
\end{proof}

\section{Conclusion}

There are no interesting examples
of geometric structures on
compact complex surfaces modelled
on Lie's exotic surfaces. The discovery
that this is so is
an essential step in the classification
of holomorphic geometric structures
on low dimensional compact complex manifolds.
Complex tori have obvious translation
structures, while Kodaira surfaces
have similar canonical geometric
structures, and all of the exotic
geometric structures on compact 
complex surfaces are induced
from these more elementary structures.
Therefore we can ignore the exotic
surfaces of Lie in the search for 
locally homogeneous holomorphic geometric structures
on compact complex surfaces. 

The
main difficulty in classifying locally
homogeneous structures is that the
model $X$ might be very flexible, i.e.
the group $G$ might be very large, and so
it becomes difficult to see the constraints
on putting together a $(G,X)$-structure
on a given surface. In particular,
since the groups acting on Lie's exotic surfaces
are of arbitrarily large dimension, the exotic
surfaces present an exceptionally difficult
case study, which we had to face before
we can develop the general theory of 
holomorphic geometric structures on
compact complex surfaces.

\bibliographystyle{amsplain}
\bibliography{toriExoticGeometry}

\end{document}